\documentclass[12pt]{amsart}
\usepackage{hyperref}
\usepackage[capitalize, nameinlink]{cleveref}
\crefname{theorem}{Theorem}{Theorems}
\crefname{proposition}{Proposition}{Propositions}
\crefname{observation}{Observation}{Observations}
\crefname{lemma}{Lemma}{Lemmas}
\crefname{claim}{Claim}{Claims}
\crefname{problem}{Problem}{Problems}
\crefname{conjecture}{Conjecture}{Conjectures}
\crefname{question}{Question}{Questions}
\crefname{example}{Example}{Examples}
\crefname{fact}{Fact}{Facts}

\usepackage{amssymb}
\usepackage[bbgreekl]{mathbbol}
\usepackage{graphicx}
\usepackage{ifthen}
\usepackage{pict2e}
\usepackage{xargs}
\usepackage{xspace}
\usepackage{xcolor}
\usepackage{pgf,tikz}
\usepackage{pgfplots}
\usepackage{environ}
\usepackage{caption}
\usetikzlibrary{arrows}
\usepackage{enumitem}
\usepackage{xifthen}
\usepackage{comment}
\newcounter{dummy}
\makeatletter
\newcommand\myitem[1][]{\item[#1]\refstepcounter{dummy}\def\@currentlabel{#1}}
\makeatother

\makeatletter
\newsavebox{\measure@tikzpicture}
\NewEnviron{scaletikzpicturetowidth}[1]{%
	\def\tikz@width{#1}%
	\begin{lrbox}{\measure@tikzpicture}%
		\BODY
	\end{lrbox}%
	\pgfmathparse{#1/\wd\measure@tikzpicture}%
	\BODY
}
\makeatother

\DeclareSymbolFontAlphabet{\mathbb}{AMSb}

\newcommand{\thistheoremname}{}
\newtheorem*{genericthm*}{\thistheoremname}
\newenvironment{namedthm*}[1]
{\renewcommand{\thistheoremname}{#1}%
	\begin{genericthm*}}
	{\end{genericthm*}}

\newcommand{\Bairespace}[1][]{
	\ifthenelse{\equal{#1}{}}{\functions{\N}{\N}}{\functions{#1}{\N}}
}
\newcommand{\bbL}{\mathbb{L}}
\newcommand{\bbX}{\mathbb{X}}

\newcommand{\Cantorspace}[1][]{
	\ifthenelse{\equal{#1}{}}{\functions{\N}{2}}{\functions{#1}{2}}
}

\newcommandx{\concatenation}[2][1 = undefined, 2 = undefined]{
	\ifthenelse{\equal{#1}{undefined}}{{}\smallfrown}{
		\ifthenelse{\equal{#2}{undefined}}{\bigoplus #1}{\bigoplus_{#1} #2}
	}
}

\newcommandx{\functions}[3][3 =]{
	\ifthenelse{\equal{#3}{}}{#2^{#1}}{#2_{#3}^{#1}}
}

\newcommand{\Gzero}[1][]{
	\ifthenelse{\equal{#1}{}}
	{\mathbb{G}_0}
	{\mathbb{G}_{0,n}}
}
\newcommandx{\Hzero}[2][2 = undefined]{
	\ifthenelse{\equal{#2}{undefined}}
	{\mathbb{H}_{#1}}
	{\mathbb{H}_{#1, #2}}
}
\newcommandx{\intersection}[2][1 =, 2 =]{
	\ifthenelse{\equal{#1}{}}{\cap}{
		\ifthenelse{\equal{#2}{}}{\bigcap #1}{{\bigcap_{#1} #2}}
	}
}
\newcommand{\Lzero}[1][]{\ifthenelse{\equal{#1}{}}{\bbL_0}{L_{0, #1}}}
\newcommand{\Lzerospace}[1][]{\ifthenelse{\equal{#1}{}}{\bbX_0}{X_{0, #1}}}

\newcommand{\modulo}[1]{\ (\text{mod } 2)}
\newcommand{\N}{\mathbb{N}}

\newcommandx{\product}[2][1 =, 2 =]{
	\ifthenelse{\equal{#1}{}}{\times}{
		\ifthenelse{\equal{#2}{}}{\prod #1}{{\prod_{#1} #2}}
	}
}

\newcommandx{\sequence}[2][2 = undefined]{
	\ifthenelse{\equal{#2}{undefined}}{(#1)}{
		(#1)_{#2}
	}
}

\newcommandx{\set}[2][2 = undefined]{
	\ifthenelse{\equal{#2}{undefined}}{\{ #1 \}}{
		\{ #1 \suchthat #2 \}
	}
}
\newcommandx{\sets}[3][3 =]{
	\ifthenelse{\equal{#3}{}}{[#2]^{#1}}{[#2]^{#1}_{#3}}
}

\newcommand{\suchthat}{\mid}

\renewcommand{\restriction}[2]{#1 \upharpoonright #2}

\newcommandx{\union}[2][1 =, 2 =]{
	\ifthenelse{\equal{#1}{}}{\cup}{
		\ifthenelse{\equal{#2}{}}{\bigcup #1}{{\bigcup_{#1} #2}}
	}
}

\newtheorem{theorem}{Theorem}[section]
\newtheorem{lemma}[theorem]{Lemma}

\newtheorem{claim}[theorem]{Claim}
\newtheorem{corollary}[theorem]{Corollary}
\newtheorem{proposition}[theorem]{Proposition}

\theoremstyle{definition}

\newtheorem{definition}[theorem]{Definition}

\numberwithin{equation}{section}

\newcommand{\bd}{\begin{definition}}
	\newcommand{\ed}{\end{definition}}

\DeclareMathOperator{\dist}{dist}
\DeclareMathOperator{\didistance}{didist}

\newcommand{\distance}[3]{\ifthenelse{\isempty{#3}}{\dist(#1,#2)}{\dist^{#3}(#1,#2)}}
\newcommand{\didist}[3]{\ifthenelse{\isempty{#3}}{\didistance(#1,#2)}{\didistance^{#3}(#1,#2)}}
\newcommand{\digraph}[3]{\ifthenelse{\equal{#1}{b}}{\mathbb{#2}_{#3}}
	{{#2}_{#3}}}
\newcommand{\linegraph}[3]{\ifthenelse{\equal{#1}{b}}{\mathbb{#2}_{#3}}
	{#2_{#3}}}

\newcommand{\underlyingspace}[3]{\ifthenelse{\equal{#1}{b}}{\mathbb{#2}_{#3}}
	{#2_{#3}}}
\newcommand{\distanceset}[2]{\ifthenelse{\isempty{#2}}{D(#1)}{D^{#2}(#1)}}

\newcommand{\SG}{\mathcal{G}_\mathcal{S}}

\newcommand{\fG}{\mathcal{G}}

\newcommand{\fQ}{\mathcal{Q}}

\newcommand{\fS}{\mathcal{S}}

\begin{document}

	\thanks{}
	
	\keywords{}
	
	\subjclass[2020]{Primary 03E15, Secondary 28A05, 05C80}
	
	\title[]{Ramsey, expanders, and Borel chromatic numbers}
	
	\author{Jan Greb\' ik}
    \address{
    Mathematics Institute.
    University of Warwick,
    Coventry CV4~7AL, UK}
    \email{jan.grebik@warwick.ac.uk}
    \author{Zolt\'an Vidny\'anszky}
    \address{Department of Mathematics, California Institute of Technology, 1200 E California Blvd, Pasadena CA 91125, USA}
    \email{vidnyanz@caltech.edu}
    \thanks{The first author was supported by Leverhulme Research Project Grant RPG-2018-424.
    The second author was supported by the National Research, Development and Innovation Office -- NKFIH, grants no.~113047, ~129211, and FWF Grant M2779.}

	
	
	\maketitle
	
	\begin{abstract}
		We construct bounded degree acyclic Borel graphs with large Borel chromatic number using a graph arising from Ramsey theory and limits of expander sequences.
	\end{abstract}

	\maketitle
	\section{Introduction}
	
	The investigation of definable versions of finite combinatorial concepts has generated a considerable interest in the past decade. The development of this field--often called descriptive combinatorics--revealed deep connections to other fields, such as the theory of countable groups, ergodic theory, FIID processes, and even distributed computing (see \cite{pikhurko2021descriptive_comb_survey,kechris_marks2016descriptive_comb_survey} for recent surveys). 
	
	
	Definable chromatic numbers play a prominent role in descriptive combinatorics. Recall that a \emph{Borel graph} $\mathcal{G}$ on a Polish space (i.e., separable, completely metrizable topological space) $X$ is a graph on $X$ with an edge set which is Borel in $X^2$. The \emph{Borel chromatic number} of $\mathcal{G}$ is the minimal $n \in \{1,2,\dots,\aleph_0\}$ such that $\mathcal{G}$ admits a Borel measurable $n$-coloring, in notation, $\chi_B(\mathcal{G})=n$. One can similarly define \emph{Baire measurable} and \emph{$\mu$-measurable} chromatic numbers requiring the coloring to be Baire measurable, or measurable with respect to some measure $\mu$ on $X$. We denote these by $\chi_{BM}$ and $\chi_\mu$, respectively.
	
	The systematical investigation of Borel chromatic numbers has been initiated by Kechris--Solecki--Todor\v{c}evi\'c \cite{KST}. They have shown that the Borel chromatic number of a graph with degrees at most $\Delta$ is bounded by $\Delta+1$. It was a long standing open question, whether this bound is sharp for $\Delta>2$, if one excludes the trivial case of complete graphs on $\Delta+1$ vertices. Eventually, Marks \cite{DetMarks} answered this question in the affirmative by constructing acyclic $\Delta$-regular Borel graphs with Borel chromatic number $\Delta+1$. Together with the results of Conley--Marks--Tucker-Drob \cite{BrooksMeas} this yielded examples of bounded degree acyclic Borel graphs $\mathcal{G}_\Delta$ with $\chi_{BM}(\mathcal{G}_\Delta),\chi_\mu(\mathcal{G}_\Delta)<\chi_B(\mathcal{G}_\Delta)$ (for any $\mu$ Borel probability measure). Marks' argument relies on one of the deepest theorems of descriptive set theory: Martin's Borel Determinacy. 
	
	Essentially all the reasons behind large Borel chromatic numbers of bounded degree acyclic Borel graphs come from the places discussed above, namely from Baire category, measure theory, or Borel determinacy arguments. In this paper we construct such graphs which have a large Borel chromatic number because a novel reason: an infinite dimensional generalization of Ramsey's theorem, the Galvin-Prikry theorem. Our technique also allows to simply transfer the complexity results established in \cite{toden} to the bounded degree acyclic setting (unlike the argument in \cite{brandt_chang_grebik_grunau_rozhon_vidnyaszkyhomomorphisms}, which relies on Marks' method). Hence, our main result is showing the following \emph{without invoking the Borel Determinacy Theorem}.
	
	\begin{theorem}
		\label{t:mainintro}
		For each $n \geq 4$ there exists a bounded degree acyclic Borel graph with $\chi_{BM}(\mathcal{H}_n), \chi_\mu(\mathcal{H}_n)\leq 3<\chi_B(\mathcal{H}_n)=n$.
		
		Moreover, the collection of Borel subgraphs of $\mathcal{H}_n$ with Borel chromatic number $<n$ is $\mathbf{\Sigma}^1_2$-complete\footnote{For the definition of such sets we refer the reader to \cite{kechrisclassical}, however, our main construction is accessible without familiarity with these concepts.}.
	\end{theorem}

	A key weakness of our construction compared to the one in \cite{DetMarks} is that the degrees of the graph are much larger than its Borel chromatic number. In fact, we do not know whether it is possible to prove the existence of $\Delta$-regular acyclic Borel graphs with Borel chromatic number $\Delta+1$ without the Borel determinacy theorem.

	Our construction uses two ingredients. First, the shift graph $\mathcal{G}_\fS$  on the space of infinite subsets of $\N$ introduced in \cite{KST}. This graph is locally finite, acyclic, and by the Galvin-Prikry theorem $\chi_B(\mathcal{G}_S)=\aleph_0$, hence it has unbounded degrees. Second, we use limits of random $\Delta$-regular finite graphs, and the theory of local-global limits \cite{hatamilovaszszegedy}. The latter argument yields the following statement, which is interesting on its own.
	\begin{proposition}
		\label{pr:expanderintro}
		For any $\varepsilon>0$ there exists a bounded degree acyclic Borel graph $\mathcal{G}$ on a Borel probability measure space $(X,\mu)$ with the following property: if $\mu(B) \geq \varepsilon$ then $\mu(N_\mathcal{G}(B)) > 1-\varepsilon$ (where $N_\mathcal{G}(B)$ is the set of neighbors of vertices in $B$). 
	\end{proposition}

		Let us point out that this expansion property resembles the property proved using Marks' game \cite{DetMarks}. Indeed, for say, $n=3$, the existence of appropriate winning strategies in a sense selects the ``largest" piece of a partition of the vertices into three Borel pieces, and the game ensures that there is an edge between any two large pieces. In our case, largeness is determined using measure (take $\varepsilon=\frac{1}{3}$), and the above statement guarantees the existence of an edge between large sets.

	We close this section by giving a high-level overview of the construction. 
	
	First, we start with the shift graph  $\mathcal{G}_\fS$ and replace each degree $d$ vertex with a $d$-tuple of vertices and distribute the edges of the shift graph so that each degree becomes $1$. Then we glue a finite graph, a ``gadget", to each $d$-tuple, which ensures that any $n$-coloring must be constant on each $d$-tuple. This produces a graph $\mathcal{H}'_n$ with bounded degree and chromatic number $>n$. 
	
	Second, we construct an acyclic graph with the required properties. The basic observation is that if $\mathcal{G}$ is a graph as in \cref{pr:expanderintro} with $\varepsilon=\frac{1}{n}$ then  $\chi_B(\mathcal{H}'_n \times \mathcal{G})>n$: indeed, if an $n$-coloring of $\mathcal{H}'_n \times \mathcal{G}$ existed, one could color a vertex $x$ of $\mathcal{H}'_n$ by the color of the majority of points in the section corresponding to $x$. While the product $\mathcal{H}'_n \times \mathcal{G}$ is not acyclic, it already contains only even cycles. To obtain an acyclic graph, we use a graph $\mathcal{G}$ with slightly larger expansion and take an acyclic subgraph of the product that has large chromatic number.

	\section{Preliminaries}
		We use $V(\mathcal{G})$ for the vertex set of the graph $\mathcal{G}$ and, abusing the notation, we will identify a graph with its edge set. If $B \subseteq V(\mathcal{G})$ then $\restriction{\mathcal{G}}{B}$ stands for the graph with vertex set $B$ and edge set $\mathcal{G} \cap B^2$. A \emph{$k$-edge coloring} of a graph $\mathcal{G}$ is a mapping from $\mathcal{G}$ to $k$ so that edges sharing a common vertex are mapped to different colors.
	
	 We use $[A]^\N$ for the collection of infinite subsets of $A$, if $A$ is countable, this is a Polish space endowed with the topology inherited from $2^A$. The \emph{shift graph} on $[\N]^\N$ is defined as the symmetrization of the graph of the \emph{shift map} $\mathcal{S}$, that is, $\mathcal{S}(x)=x \setminus \{\min x \}.$
	 
	As mentioned above, the Galvin-Prikry theorem is an infinite dimensional definable generalization of Ramsey's theorem. It states that if $A_0 \cup \dots \cup A_n=[\N]^\N$ is a partition to Borel sets, then there is an $x \in [\N]^\N$ and an $i \leq n$ with $[x]^\N \subseteq A_i$. Clearly, this implies that $\chi_B(\mathcal{G}_\mathcal{S}) \geq \aleph_0$. 
	
	\begin{claim}
	\label{cl:triv}
    $\chi_{BM}(\mathcal{G}_\mathcal{S}),\chi_\mu(\mathcal{G}_\mathcal{S}) \leq 3$ for any Borel probability measure $\mu$ on $[\N]^\N$.
	\end{claim}
	 \begin{proof}
	     By \cite[Lemma 4.5]{toden} $\mathcal{G}_
		\fS$ admits a Borel $3$-coloring restricted to so called non-dominating sets. It is well-known (see, e.g., \cite{bartoszynski1995set}) that there are non-dominating sets that are co-meager, and that for any $\mu$ there are such $\mu$-measure $1$ sets as well. 	
	 \end{proof}

	For further properties of the shift graph see, e.g., \cite{KST,toden,di2015basis}. 
	
	Let $\mu$ be a Borel probability measure on the space $X$. Recall that a \emph{Borel graphing} on $(X,\mu)$ is a Borel graph on $X$ that can be expressed as the union of the graphs of countably many $\mu$-preserving Borel partial involutions. 
	
	In \cref{sec:Exp} we use the concept of \emph{colored statistics} from \cite{hatamilovaszszegedy}. The definition of this notion involves assigning colors to vertices of graphs in ways which are not necessarily colorings in our sense, i.e., not proper colorings. To overcome this slightly unfortunate terminological issue, we will use the term \emph{color assignment} for such mappings.
	
	\section{Large chromatic number from Galvin-Prikry}

	\begin{proposition}
		\label{pr:boundeddegree} Assume that $n\geq 3$ is given. There exists a bounded degree Borel graph $\mathcal{H}'_n$ on $[\N]^\N \times \N$ such that for each Borel set $B \subseteq [\N]^\N$ we have \[\chi_B( \restriction{\SG}{B}) \leq 3 \iff \chi_B(\restriction{\mathcal{H}'_n}{B \times \N}) 
		\leq n.\]
	\end{proposition}
	We will need some simple gadgets.
	\begin{lemma}
		\label{l:gadget}
		Let $n \geq 3$ and $d\geq 1$. There exist an $l_{n,d} \geq d$ and a finite graph $F_{n,d}$ with $V(F_{n,d})=\{0,\dots,l-1\}$ such that 
		\begin{itemize}
			
			\item the maximum degree of $F_{n,d}$ is $2(n-1)$,
			\item if $c$ is an $n$-coloring of $F_{n,d}$, then $c(0)=\dots=c(d-1)$,
			\item there exists an $n$-coloring $c_{n,d}$ of $F_{n,d}$ with $c(0)=\dots=c(d-1)=0$.
		\end{itemize}
		
	\end{lemma}
	
	\begin{proof}
		Let $l=d+(d-1)(n-1)$ and define $F_{n,d}$ on $l \setminus d$ so that it consists of $d-1$-many disjoint copies $(K_i)_{i<d-1}$ of the complete graph on $n-1$ vertices. For each $i<d-1$ connect $i$ and $i+1$ to all the vertices of $K_i$. It is easy to verify that $F_{n,d}$ has the required properties. 
	\end{proof}
	\begin{proof}[Proof of \cref{pr:boundeddegree}]
		Define $\mathcal{H}'_{n}$ on $[\N]^\N \times \N$ as follows. Let $(x,i), (y,j) \in [\N]^\N \times \N$ be adjacent if 
		\begin{itemize}
			\item $\mathcal{S}(x)=y$, $i=\deg_{\SG}(x)-1$, and $j=\min x$ or 
			\item as above, with the roles of $x$ and $y$ reversed or
			\item $x=y$ and $(i,j) \in F_{n,d}$, where $d=\deg_{\SG}(x)$. 
		\end{itemize} 
		
		Since the function $\deg_{\SG}$ is continuous on $[\N]^\N$, the graph $\mathcal{H}'_n$ is Borel. We show that it has the required properties. First, it is clear from the definition that the degrees of $\mathcal{H}'_n$ are bounded by $2n-1$.  
		
		Let $B \subseteq [\N]^\N$ be Borel, and assume that $\chi_{B}(\restriction{\SG}{B}) \leq 3$ is witnessed by the coloring $c$. Then it is easy to define a Borel $n$-coloring of $\restriction{\mathcal{H}'_n}{B \times \N}$, e.g., let \[c'(x,i)= \begin{cases}
		c(x), \text{ if $i<\deg_{\SG}(x)$},\\
		c_{n,d}(i)+c(x) \mod n, \text{ if } l_{n,d}> i \geq \deg_{\SG}(x),\\
		0, \text{ otherwise}.
		\end{cases}\]
		
		Conversely, if $c'$ is a Borel $n$-coloring of $\restriction{\mathcal{H}'_n}{B \times \N}$, then, by the properties of $F_{n,d}$ we have that $c'$ is constant on $\{(x,i):i<\deg_{\SG}(x)\}$. Let $c(x)$ be this constant value for every $x \in B$. Then $c$ is Borel, and if $\mathcal{S}(x)=y$ and $j=\min x$ then the pair $(x,\deg_{\SG}(x)-1)$ and $(y,j)$  form an edge in $\mathcal{H}'_n$. As $j<\deg_{\SG}(y)$, it follows that $c(x) \neq c(y)$. So $c$ is a Borel finite coloring of $\restriction{\SG}{B}$, which yields the existence of a Borel $3$-coloring by \cite[Theorem 5.1]{KST}. 
	\end{proof}

	Using a less trivial argument, we can make the graph acyclic. In order to do this, we need a collection of acyclic graphs with large expansion.
	The proof of the following theorem is deferred to \cref{sec:Exp}.

	\begin{theorem}
		\label{t:auxiliary}
		Let $k \geq 1$ and $n \geq 3$. There exist disjoint Borel graphs $\mathcal{G}_j$ for $j<k$ on a probability measure space $(Y,\mu)$ such that 
		\begin{enumerate}
			\item $\bigcup_{j<k} \mathcal{G}_{j}$ is acyclic and has bounded degree.
			\item \label{c:exp} For every $j<k$ if $B, B' \subseteq Y$ are measurable and $\mu(B),\mu(B') \geq \frac{1}{n}$ then there exist $z \in B$ and $z' \in B'$ that are adjacent in $\mathcal{G}_j$.
		\end{enumerate}
	\end{theorem}

	Using this result, and a product type construction, we will be able to assemble the promised acyclic graph. Observe also that it implies \cref{pr:expanderintro} by letting $\frac{1}{n}<\varepsilon$ and $\mathcal{G}=\mathcal{G}_0$.  
	
	\begin{proposition}
		\label{pr:acyclic}
		Let $\mathcal{G}$ be a Borel graph on a space $X$ with a Borel edge $k$-coloring and $\chi_B(\mathcal{G})>n$. Fix the graphs $(\mathcal{G}_{j})_{j<k}$ and $(Y,\mu)$ provided by \cref{t:auxiliary} applied to $k$ and $n$. There exists a acyclic bounded degree Borel graph $\mathcal{H}$ on $X \times Y$ with the property that for each $B \subseteq X$ Borel 
		\[\chi_B( \restriction{\mathcal{G}}{B}) \leq n \iff \chi_B(\restriction{\mathcal{H}}{B \times Y}) 
		\leq n.\]
	\end{proposition}
	
	Our graph $\mathcal{H}$ will be an acyclic subgraph of the (categorical) product of $ \mathcal{G} \times (\bigcup_{j} \mathcal{G}_j)$. The expansion property will ensure the large chromatic number of $\mathcal{G}$, while $\mathcal{G}_j$ together with a proper edge-coloring of $\mathcal{H}$ will be used to show that the graph is acyclic.
	
	\begin{proof}

		Fix a Borel edge $k$-coloring $e$ of the graph $\mathcal{G}$. The graph $\mathcal{H}$ is defined on the space $X \times Y$ as follows: let \[(x,y)\mathcal{H}(x',y') \iff x\mathcal{G}x', e(x,x')=j \text{ and } y\mathcal{G}_jy'.\]
		It is clear that $\mathcal{H}$ has degrees bounded by $kd$, where $d$ is the maximum degree in $\mathcal{G}_j$ for $j<k$. 
		
		\begin{lemma}
			$\mathcal{H}$ is acyclic.
		\end{lemma}
		\begin{proof}
			If $(x_0,y_0),\dots,(x_n,y_n)=(x_0,y_0)$ was an injective cycle in $\mathcal{H}$, then $y_0,\dots,y_n$ would be a closed walk in $\bigcup \mathcal{G}_j$. Since $\bigcup \mathcal{G}_j$ is acyclic, there must be an $i$ with $y_{i-1}=y_{i+1}$. As $e$ is a proper edge coloring of $\mathcal{G}$ and the cycle is injective, it follows that $j=e(x_{i-1},x_{i}) \neq e(x_i,x_{i+1})=j'$. Thus, $\mathcal{G}_j \ni \{y_{i-1},y_i\}=\{y_{i},y_{i+1}\} \in \mathcal{G}_{j'}$, contradicting the assumption on disjointness.
		\end{proof}
		
		Finally, let $B \subseteq X$ be an arbitrary Borel set. If $\chi_B(\restriction{\mathcal{G}}{B}) \leq n$ then the coloring witnessing this can be extended to a Borel $n$-coloring of $\restriction{\mathcal{H}}{B}$, by making it constant on every set of the form $\{x\} \times Y$. 
		
		Conversely, assume that $c'$ is a proper Borel $n$-coloring of $\restriction{\mathcal{H}}{B \times Y}$. Using that $c'$ is Borel and \cite[Theorem 29.26]{kechrisclassical} we obtain that the set \[M=\{(x,i):\mu(\{y \in Y:c'(x,y)=i\}) \geq \frac{1}{n}\}\] is Borel as well. Then there is a Borel map $c:B \to n$ (e.g., letting $c(x)$ be the minimal $i$ with $(x,i) \in M$) such that  \[c(x)=i \implies \mu(\{y \in Y:c'(x,y)=i\}) \geq \frac{1}{n}.\]
		We show that $c$ is a proper $n$-coloring of $\restriction{\mathcal{G}}{B}$. Let $x\mathcal{G}x'$ and $x,x' \in B$. Then, for some $j<k$ we have $e(x,x')=j$. It follows from Property \eqref{c:exp} of $\mathcal{G}_j$ and the definition of $c$ that there are $z \in \{y \in Y:c'(x,y)=c(x)\}$ and $z' \in \{y \in Y:c'(x',y)=c(x')\}$ that are adjacent in $\mathcal{G}_j$. Consequently, $(x,z) \mathcal{G} (x',z')$, so $c(x)=c'(x,z) \neq c'(x',z')=c(x')$.
	\end{proof}
	Now we have the following, which encompasses \cref{t:mainintro}.
	\begin{theorem}
		\label{t:acyclic} Assume that $n\geq 3$ is given. There exists a Polish space $X$ and a bounded degree acyclic Borel graph $\mathcal{H}$ on $[\N]^\N \times X$ such that for each Borel set $B \subset [\N]^\N$ we have \[\chi_B( \restriction{\SG}{B}) \leq 3 \iff \chi_B(\restriction{\mathcal{H}}{B \times X}) 
		\leq n.\]
		
		In particular, codes of Borel subgraphs of $\mathcal{H}$ with Borel chromatic number $\leq n$ form a $\mathbf{\Sigma}^1_2$-complete set.
		
		Moreover, $\chi_{BM}(\mathcal{H}),\chi_\mu(\mathcal{H}) \leq 3$ for any Borel probability measure $\mu$ on $[\N]^\N \times X$.
	\end{theorem}
	\begin{proof}
		The graph $\mathcal{H}'_n$ constructed in \cref{pr:boundeddegree} has bounded degrees, so it admits a Borel $k$-edge coloring for some $k$ by \cite[Proposition 4.6]{KST}. Now applying \cref{pr:acyclic} to $n,k$ and this graph yields the graph $\mathcal{H}$ with the desired property. 
		
		The second part of the statement follows from the main result of \cite{toden}, namely that codes of Borel sets for which $\chi_B(\restriction{\SG}{B}) \leq 3$ form a $\mathbf{\Sigma}^1_2$-complete set and the observation that the map $B \mapsto B \times X$ is Borel in the codes.  
		
		In order to see the last statement, using \cref{cl:triv} for the pushforward of a $\mu$ on $[\N]^\N \times X$ w.r.t. the projection map to $[\N]^\N$, we obtain $\chi_\mu(\mathcal{H})=3$. Similarly, $\chi_{BM}(\mathcal{H})=3$ follows by taking a co-meager set $B \subseteq [\N]^\N$ with $\chi_{BM}(\restriction{\mathcal{G}_\fS}{B}) \leq 3$ and noting that $B \times X$ is co-meager.
	\end{proof}

	\section{Proof of \cref{t:auxiliary}}\label{sec:Exp}
	
	There are three technical ingredients in the proof: (1) we use the expander mixing lemma, (2) we use a standard fact that random regular graphs are expanding asymptotically almost surely, and (3) we use the theory of \emph{local-global} limits to find the desired graphing.
	
	Given a finite $d$-regular graph $G$, we write $\lambda_G$ for the second largest absolute value of the eigenvalues of the adjacency matrix (as $G$ is $d$-regular, the largest is always $d$). 
	For $A,B\subseteq V(G)$, write $e(A,B)$ for the number of edges from $A$ to $B$ (where edges spanning vertices from $A\cap B$ are counted twice). The next classical lemma says that if $\lambda_G$ is small, then the number of edges between two sets of vertices is not far from the number one would expect based on their relative size.
	
	\begin{lemma}[Expander Mixing Lemma Corollary~9.2.5 in \cite{alon2016probabilistic}]\label{lm:EML}
		Let $G$ be a $d$-regular graph and $B,B'\subseteq V(G)$.
		Then we have
		$$\left|e(B,B')-\frac{d|B||B'|}{|V(G)|}\right|\le \lambda_G \sqrt{|B||B'|}.$$
	\end{lemma}	
	Fix $n$ and $k$ as in \cref{t:auxiliary}. 
	
	\begin{corollary}\label{cor:numberofedgs}
		If $\lambda_G\le \frac{d}{n+2}$, then $e(B,B')>\frac{d|V(G)|}{(n+1)^2(n+2)}$, whenever $|B|,|B'|>|V(G)|/(n+1)$.
	\end{corollary}
	\begin{proof}
		By \cref{lm:EML}, we have
		\begin{equation*}
			\begin{split}
				\left|e(B,B')\right|\ge & \ \frac{d|B||B'|}{|V(G)|}-\lambda_G \sqrt{|B||B'|} \\
				= & \ \sqrt{|B||B'|}\left(\frac{d\sqrt{|B||B'|}}{|V(G)|}-\lambda_G\right) \\
				> & \ \frac{|V(G)|}{(n+1)}\left(\frac{d}{n+1}-\frac{d}{n+2}\right)\ge \frac{d|V(G)|}{(n+1)^2(n+2)}
			\end{split}
		\end{equation*}

	\end{proof}
	
	We work with the following random graph model.
	Let $\ell$ be even and $d\in \mathbb{N}$.
	Write $\mathbb{G}_{\ell,d}$ for the $d$-regular random graph on $\ell$ vertices that is obtained by sampling $d$-many perfect matchings independently uniformly at  random.
	Note that we allow parallel edges.
	
	\begin{theorem}\label{th:spectral}
		Let $G\sim\mathbb{G}_{\ell,d}$ be a random $d$-regular graph.
		Then asymptotically as $\ell \to \infty$ almost surely we have $\lambda_G\le \frac{d}{n+2}$, for every $d$ large enough (depending on $n$).
	\end{theorem}
	\begin{proof}
		This follows from a much stronger result \cite[Theorem~1.3]{FriedmanExpander}, namely that $\lambda_G$ is bounded above by $2 \sqrt{d-1}+1$, which is of course $<\frac{d}{n+2}$ for every large enough $d$. 
	\end{proof}
	
	Next we build a sequence of graphs $(G_\ell)_{\ell}$ that will approximate the desired graphing $\fG$.
	Fix $\ell\in\mathbb{N}$ even.
	Consider $k$ independent samples from $\mathbb{G}_{\ell,d}$, where $d\in\mathbb{N}$ satisfies \cref{th:spectral}.
	We denote by $G^j_\ell$ the $j$-th sample and set $G_\ell=\bigcup_{j<k} G^j_\ell$.
	\begin{proposition}\label{pr:sequence}
		The sequence $(G_\ell)_\ell$ has the following properties with non-zero probability:
		\begin{enumerate}
			\item $\lambda_{G^j_\ell}\le \frac{d}{n+2}$ for every $n$ and $j<k$,
			\item the girth of $G_\ell$ tends to $\infty$ as $\ell\to \infty$.
		\end{enumerate}
	\end{proposition}
	\begin{proof}
		The first assertion follows from \cref{th:spectral}.
		The second from the fact that $G_l\sim \mathbb{G}_{\ell,kd}$ and the fact that random regular graphs have large girth with non-zero probability \cite{bollobas,mckayshortcycles}.
	\end{proof}
	
	Fix such a sequence $(G_\ell)_\ell$.
	We note that instead of (2) in \cref{pr:sequence}, we could have only required that for each $m\in \mathbb{N}$ there is at most $o(\ell)$ cycles of length $m$.
	This holds asymptotically almost surely \cite{bollobas}.
	
	To finish the proof of \cref{t:auxiliary}, we use the theory of local-global convergence of bounded degree graphs \cite{hatamilovaszszegedy}.
	In order to do that, we need to recall the notion of \emph{colored statistics} of a finite graph or a graphing.
	Fix $\Delta\in \mathbb{N}$ and write $\mathfrak{G}$ for the class of connected graphs of degree bounded by $\Delta$, in our case $\Delta=kd$.
	Let $r,s\in \mathbb{N}$.
	Write $\mathfrak{G}^{\bullet}_{r,s}$ for the space of  $s$-color assigned rooted $r$-neighborhoods of vertices from graphs from $\mathfrak{G}$.
	That is, an element $\mathfrak{G}^{\bullet}_{r,s}$ is a rooted graph endowed with a vertex color assignment with $s$-colors that that can be found as an $r$-neighborhood of some vertex of some graph in $\mathfrak{G}$.
	Note that as the degree is uniformly bounded, each set $\mathfrak{G}^{\bullet}_{r,s}$ is finite (up to isomorphism).
	Let $G\in \mathfrak{G}$ be a finite graph, $r,s\in \mathbb{N}$ and $\beta:V(G)\to s$ be a vertex color assignment.
	The \emph{$(\beta,r)$-statistic} is a probability distribution on the space $\mathfrak{G}^{\bullet}_{r,s}$ that can be produced by sampling uniformly a vertex of $G$ and looking at the restriction of $\beta$ to its $r$-neighborhood.
	Write $\fQ_{G,r,s}$ for the space of all $(\beta,r)$-statistics, where $\beta$ runs over all vertex $s$-color assignments of $G$.
	By considering \emph{measurable} vertex $s$-color assignments, one can define the set $\fQ_{\fG,r,s}$, and all the other concepts, for any graphing $\fG$ of degree bounded by $\Delta$.
	
	In our application, the graphs are endowed with a fixed edge labeling with finitely many labels\footnote{Formally, there is no difference between \emph{labelings} and \emph{color assignments}. We use the word ``labeling'' to indicate that the function is a fixed part of the structure of the graph, while ``color assignment'' if it varies e.g. as in the definition of $\fQ_{G,s,r}$.} (namely, the labelings corresponding to the partition $G_\ell=\bigcup_j G^j_\ell$).
	As discussed in \cite[Section~19.1.4]{LovaszLargeNetworks}, the whole theory of local-global convergence works for bounded degree graphs endowed with such an additional structure.
    In that case, $\mathfrak{G}$ is the space of all connected graphs of degree bounded by $\Delta$ that are additionally endowed with an edge $k$-labeling, where $k\in \mathbb{N}$ is fixed.
    Similarly, $\mathfrak{G}^{\bullet}_{r,s}$ stands for the space of all vertex $s$-color assigned rooted edge $k$-labeled $r$-neighborhoods of vertices from graphs from $\mathfrak{G}$, and the $(\beta,r)$-statistics are with respect to the fixed edge $k$-labeling of the graph $G$, or measurable edge $k$-labeling of a graphing $\fG$.
    Altogether, we have the following result.
	
	\begin{theorem}[Theorem~3.2 in \cite{hatamilovaszszegedy} or Theorem~19.16 in \cite{LovaszLargeNetworks}]
		Every sequence of bounded degree graphs $(G_\ell)_\ell$ contains a subsequence that converges locally-globally to a graphing $\fG$.
		That is, if we write $(\ell_m)_{m\in \mathbb{N} }$ for the subsequence, then $\mathcal{Q}_{G_{\ell_m},r,s}$ converges in the Gromov-Hausdorff distance to $\mathcal{Q}_{\fG,r,s}$ for every $r,s\in \mathbb{N}$.
		
		A similar result holds when we assume that $(G_\ell)_{\ell\in \mathbb{N}}$ comes with a fixed edge labeling.
		In that case, $\fG$ is endowed with a measurable edge labeling as well and the convergence is with respect to colored statistics of edge labeled neighborhoods.
	\end{theorem}
	
	Now it is easy to finish the proof of \cref{t:auxiliary}. Label the edges of $G^j_\ell$ by $j$. By the theorem above, by passing to a subsequence, we may assume that $(G_\ell)_\ell$ admits a local global limit, $\mathcal{G}$.
	By the definition, $\fG$ is a graphing on some probability space $(X,\mu)$ with measurable edge labeling with $k$ colors.
	Let $\fG_j$ be the restriction of $\fG$ to the $j$-th color.
	
	It follows from the ordinary local convergence that $\fG$ is acyclic, $kd$-regular and each $\fG_j$ is $d$-regular.
	Hence (1) in \cref{t:auxiliary} is satisfied.
	Finally, let us verify \eqref{c:exp}.
	Let $B,B'\subseteq X$ be such that $\mu(B),\mu(B')\geq 1/{n}$.
	In particular, we have $\mu(B),\mu(B')> 1/{(n+1)}$.
	We encode the sets $B$ and $B'$ by a measurable vertex color assignment $\beta$ with four colors, depending on which of the sets $B \setminus B',B' \setminus B,B \cap B', X\setminus (B \cup B')$ a given element belongs to.
	By the local-global convergence, we find vertex color assignments $(\alpha_\ell)_\ell$ of $(G_\ell)_\ell$ so that the colored $(\alpha_\ell,2)$-statistics of $G_\ell$ converge to the colored $(\beta,2)$-statistic of $\fG$.
	Note that $(\alpha_\ell)_\ell$ encodes sets $(B_\ell)_\ell$ and $(B'_\ell)_{\ell}$ that eventually have size strictly greater than $|V(G_\ell)|/{(n+1)}$ as this is encoded by $(\alpha_\ell,2)$-statistics (in fact $(\alpha_\ell,1)$-statistics).
	By \cref{cor:numberofedgs}, we have $e(B_\ell,B'_\ell)>d|V(G_\ell)|/{(n+1)^2(n+2)}$ in $G_n^j$ for every $j<k$ and $\ell$ large enough.
	Since the degrees of each $G^j_\ell$ are bounded by $d$, we have $|\{x\in B_\ell:\exists y\in B'_\ell \ (x,y)\in G^j_\ell\}|>|V(G_\ell)|/{(n+1)^2(n+2)}$.
	As this information is encoded by $(\alpha_\ell,2)$-statistics, we have that $\mu(\{x\in B:\exists y\in B' \ (x,y)\in \fG_j\})\ge 1/{(n+1)^2(n+2)}>0$ for every $j<k$.
	This finishes the proof.

	\bibliographystyle{abbrv}
	\bibliography{ref.bib}
	
\end{document}